\documentclass[11pt,noindent]{article}
\usepackage{latexsym,amsmath,amsfonts}
\newtheorem{theorem}{Theorem}

\newtheorem{proposition}{Proposition}
\newtheorem{corollary}{Corollary}
\newtheorem{lemma}{Lemma}

\newtheorem{claim}{Claim}

\numberwithin{theorem}{section}
\numberwithin{lemma}{section}
\numberwithin{proposition}{section}
\numberwithin{corollary}{section}
\numberwithin{claim}{section}

\newcommand{\thmref}[1]{Theorem~\ref{thm:#1}} % Theorem tag equals ``thm''
\newcommand{\eqnref}[1]{(\ref{eq:#1})} % Equation = ``eq''
\def\be{\begin{equation} }
\def\ee{ \end{equation}}

\def\ben{\begin{equation*}}
\def\een{\end{equation*}}
\def\bea{\begin{eqnarray}}
\def\eea{\end{eqnarray}}
\def\ee{\end{eqnarray}}
\def\bean{\begin{eqnarray*}}
\def\eean{\end{eqnarray*}}

\newcommand\ignore[1]{}

\def\R{\mathbb{R}} % real numbers
\def\N{\mathbb{N}} % naturals
\newcommand{\geo}{\text{Geo}}
\newcommand{\nh}{\text{NH}}
\newcommand{\Ex}[1]{\mathbb{E}\left[#1\right]} % \Ex{abc} prints E[abc] with appropriately sized bracketss
\newcommand{\Prwo}{\mathbb{P}} % \Prwo yields P
\renewcommand{\Pr}[1]{\mathbb{P}\left(#1\right)} % \Pr{abc} prints P(abc) with appropriately sized parentheses
\newcommand{\liloh}[1]{o\left(#1\right)}

\def\sF{\mathcal{F}}

\def\sJ{\mathcal{J}}
\def\sN{\mathcal{N}}

\def\sT{\mathcal{T}}

\newcommand\QED{\ifhmode\allowbreak\else\nobreak\fi
\quad\nobreak$\Box$\medbreak}
\newcommand{\proofstart}{\par\noindent\sl Proof:\rm\enspace}
\newcommand{\proofend}{\QED\par}
\newenvironment{proof}{\proofstart}{\proofend}

\def\eps{\epsilon}
\begin{document}

\title{Rumor Spreading on Percolation Graphs}
\author{Roberto I. Oliveira\footnote{Supported by a {\em Bolsa de Produtividade em Pesquisa} and a {\em Pronex} grant from CNPq, Brazil.} \and Alan Prata\footnote{Supported by a doctoral scholarship from CNPq, Brazil.}} \maketitle
\begin{abstract} We study the relation between the performance of the randomized rumor spreading (push model) in a $d$-regular graph $G$ and the performance of the same algorithm in the percolated graph $G_p$. We show that if the push model successfully broadcast the rumor within $T$ rounds in the graph $G$ then only $(1 + \epsilon)T$ rounds are needed to spread the rumor in the graph $G_p$ when $T = \liloh{pd}$.
\end{abstract}

\section{Introduction} {\it Randomized rumor spreading} or {\it randomized broadcasting} are simple randomized algorithms to spread information in networks.
In this work we consider the classical {\it push model} for rumor spreading, which is described as follows. Initially, one arbitrary vertex knows an information. In the succeding time steps each informed vertex chooses a neighbor independently and uniformly at random and forward the information to it. The fundamental question is: how many time steps are needed until every vertex of the network has been informed?

The push model has been extensively studied. Most of the papers analyze the runtime of this algorithm on different graph classes. On the complete graph, Frieze and Grimmett \cite{Fri} proved that {\it with high probability\footnote{With high probability, also denoted w.h.p., refers to an event that holds with probability $1 - \liloh{1}$ as the size of the graph tends to infinity.}} $(1 + \liloh{1})(\log_2n + \log{n})$ rounds are necessary and sufficient to inform all vertices. In \cite{Feige} Feige {\it et al.} gave general upper bounds holding for any graph and determine the runtime on random graphs. Also, Chierichetti, Lattanzi and Panconesi \cite{Chi} proved runtime bounds in terms of the conductance and Sauerwald and Stauffer \cite{Sau} obtain bounds in terms of the vertex expasion for regular graphs.

The starting point of this paper is a recent article by Fountoulakis, Huber and Panagiotou \cite{Foun} which analyze the push protocol on the Erd\"{o}s-R\'{e}nyi random graph $G_{n, p}$ where $p\gg\frac{\ln n}{n}$. Among other things, they show that the protocol will inform every vertex in $(1+\liloh{1})(\log_2n +\ln n)$ steps w.h.p.. One may restate this result by  the runtime of randomized broadcasting on a complete graph is essentially not affected by random edge deletions, at least up to the connectivity threshold $p=\ln n/n$. 

In this paper we prove a partial extension of this result to the case of arbitrary percolation graphs. Here one starts with some arbitrary graph $G$ and performs edge percolation on it. Under certain conditions, we show that this will not increase the runtime of the protocol by more than a $1+\liloh{1}$ factor. This suggests that the push protocol is robust against random edge failures, which is a desirable quality for applications.

We need some definitions in order to state our main Theorem. Given a graph $G=(V,E)$, we let $G_p$ denote the random subgraph of $G$ where each edge is removed independently with probability $1-p$ (the vertex set stays the same). We let $\sT_v(G),\sT_v(G_p)$ denote the runtimes of the push protocol starting at $v$ over $G$ and $G_p$ (respectively). 

\begin{theorem}\label{thm:push} Let $G_n = (V_n, E_n)$ be a sequence of $d_n$-regular graph where $|V_n| = n \rightarrow \infty$ and $v_n\in V_n$. Suppose that exists $T_n$ such that $\sT_{v_n}(G_n) \leq T_n$ with high probability. Then, given any $\epsilon > 0$, if $0<p_n<1$ is such that $T_n = \liloh{p_n.d_n}$ we have that $\sT_{v_n}(G_{n, p_n}) \leq (1 + \epsilon)T_n$ with high probability (here the probability is over the choice of $G_{n,p_n}$ and over the additional randomness of the push protocol).\end{theorem}

One can check in our proofs that the same result holds if $G_n$ has minimum degree $d_n$. The case where $G_n=K_n$ is complete shows that the condition $T_n=\liloh{p_nd_n}$ cannot be removed in general. We also remark that proving lower bounds for $\sT_{v_n}(G_{n,p_n})$ in terms of $\sT_{v_n}(G_n)$ is an interesting open problem, but the upper bound we give seems more interesting for applications

\subsection{Proof ideas}

The proof strategy of \cite{Foun} relies on the geometry of the Erd\"{o}s-R\'{e}nyi graph above the connectivity threshold. We have no such information available in our general setting, and instead rely on a very different proof strategy:

\begin{quotation}\noindent{\bf Proof strategy:} Construct a coupling of:
$$\mbox{[push over $G_n$]} \leftrightarrow \mbox{[random choice of $G_{n,p}$ $+$ push over $G_{n,p_n}$.]}$$ 
Then show that the runtimes of push over the two graphs are close under the coupling.\end{quotation}

It is not hard to sketch a coupling that solves the analogous problem over oriented graphs. Assume $D=(V,F)$ is a $n$-vertex {\em digraph}. Define $D_p$ as the random digraph obtained from $D$ by deleting each oriented edge with probability $1-p$. We consider a variant of push over $D$ and $D_p$, where each informed vertex $v$ pushes the rumour along outgoing edges chosen uniformly but {\em without replacement}. We will assume that all vertices have the same out-degree $d$ and that this modified push protocol over $D$ typically takes $T\ll pd$ steps.

We now couple 
$$\mbox{[push over $D$ up to time $T$]} \leftrightarrow \mbox{[random choice of $D_{p}$ $+$ push over $D_p$ up to time $T$.]}$$ 
For this we need two Bernoulli (indicator) random variables $A_{v\to w}$ and $I_{v\to w}$ for each oriented edge $v\to w$. All of those variables will be assumed independent, and we take:
$$\Pr{A_{v\to w}=1}=p\mbox{ and }\Pr{I_{v\to w}=1} = \frac{CT}{pd}\mbox{ for some $C>0$.}$$
Notice that $T\geq \log_2 n$ as the number of informed vertices can only double at each time step. Chernoff bounds (see \cite{AS}) imply that, if $C$ is sufficiently large, then w.h.p., for all $v\in V$, the set 
$$\sN(v)\equiv\{w\in V\,:\,v\to w\in F,\, A_{v\to w} I_{v\to w} = 1\}$$
will have at least $T$ elements, as each of the $d$ out-neighbors of $v$ belongs to this set with probability $CT/d$. This means we can run modified push over $D$ by having each $v$ select the edges $v\to w$ with $w\in \sN(v)$ in a random order, up to time $T$. Notice that this gives the right distribution because, conditionally on $|\sN(v)|=k\geq T
$, $\sN(v)$ is uniform over all $k$-subsets of out-neighbors of $v$ in $D$.

We now let $D_p$ be the digraph whose oriented edges are the pairs $v\to w$ with $A_{v\to w}=1$. To run modified push on $D_p$, have each $v$ select the edges $v\to w$, $w\in\sN(v)$, in the same order as in the protocol over $D$. The key points are that:
\begin{itemize}
\item This gives the right distribution because conditionally on $D_p$ and on $|\sN(v)|=k$, $\sN(v)$ is uniform over $k$-subsets of the out-neighbors of $v$ in $D_p$.
\item The set of informed vertices in $D$ and $D_p$ coincide up to time $T$. In particular, if all vertices are informed in $D$ up to time $T$ w.h.p., the same will hold in $D_p$.
\end{itemize}
This shows that the modified push protocol cannot take longer in $D_p$ than its typical runtime in $D$. 

As presented, this proof strategy cannot work for non-oriented graphs. The first problem is that the neighbors to be informed in the push protocol are chosen with replacement. This, however, is not hard to deal with; see Proposition \ref{L1} below. 

A second and more serious problem is that, if one tries to copy the above coupling, the events $w\in \sN(v)$ and $v\in \sN(v)$ will be positively correlated given $\{v,w\}\in G_{n,p_n}$. The solution to this will be to introduce a few {\em extra steps} in the push protocol over $G_{n,p_n}$. This is a kind of sprinkling idea. The upshot will be that the set of neighbors of $v$ chosen in push-over-$G_{n,p_n}$ will dominate the set chosen in push-over-$G$, but the diffeence between the two sets will be so small that this will not matter much. (Incidentally, this is where the $\eps$ in the Theorem comes from.)

\subsection{Organization}

This paper is organized as follows. In Section 2 we describe our model more formally,  introduce some basic notation and prove preliminary results that will be used in the proof of the main result. In Section 3 we presente the proof of the \thmref{push}. Some of the results stated in Section 2 are proved in the Appendix.

\section{Preliminaries}

Let $G = (V, E)$ be an unweighted, undirect, simple and connected graph, where $V$ is the set of vertices, $E$ the set of edges and $|V|$ denotes the size of the graph. We consider families of graphs $G_n = (V_n, E_n)$ where $|V_n| = n$. For a vertex $v \in V$, ${\rm deg} (v)$ denotes the degree of $v$ and $\Gamma(v)$ denotes the set of neighbors of $v$.

Given some parameter $p \in [0, 1]$, we consider bond percolation in $G$ by removing each edge of G, independently and with probability $1 - p$. The graph obtained from this process is denoted by $G_p$. Also, ${\rm deg}_p(v)$ denotes the degree of $v$ in $G_p$ and $\Gamma_p(v)$ denotes the set of neighbors of $v$ in $G_p$.

As mentioned before, we consider the randomized broadcasting algorithm called push algorithm. Formally, the process just explained can be described as follows. Let $I(t)$ be the set of vertices informed at time t. Initially $t = 0$ and $I(0) = \{v\}$, for some choice of $v \in V$. While $I(t) \neq V$, each vertex $u \in I(t)$ chooses a neighbor $v_u^t$ independently and uniformly at random. The new informed set is
$$I(t + 1) = I(t) \cup \{v_u^t;u \in I(t)\}.$$ The process stops when $I(t) = V$.

We are interested in how many time steps are needed for the process to stop. For this, define $\sT_v (G) = \min \{t \in \mathbb{N} | I(t) = V\}$ as the first time step in which every vertex of $G$ has been informed.

A concept used throughout this work is stochastic domination. If $X$ and $Y$ are random variables taking values in $\R$, we say that $X$ is {\it stochastically smaller} than $Y$ and denote this by $X \preceq Y$ if $\Pr{X\geq t} \leq \Pr{Y\geq t}$ for any $t \in \R$.

If $X$ has distribution $\mu$ and $Y$ has distribution $\nu$, then $X \preceq Y$ if and only if
$$\int f \,d\mu \leq \int f \,d\nu$$
for any continuous and increasing function $f$. By approximation, $X \preceq Y$ implies the relation above for all increasing upper semicontinuous function $f$.

In this paper we write $X \sim Y$ when the random variables $X$ and $Y$ have the same distribution. We let $Geo(p)$, $Be(p)$, $Bin(n, p)$ denote, respectively, the {\it geometric distribution} with parameter $p$, the {\it Bernoulli distribution} with parameter $p$ and the {\it binomial distribution} with parameters $n$ and $p$.

The rest of this section is devoted to the study of the negative hypergeometric distribution. A negative hypergeometric random variable $X$ records the waiting time in trials until the $r$-th sucess is obtained in repetead random sampling without replacement from a dichotomous population of size $N$ with $d$ successes. In the following we show that a negative hypergeometric random variable is stochastically smaller than a sum of geometric random variables. After that we can bound the probability that the sum of negative hypergeometric distribution deviates above the mean using concentration inequalities for the sum of geometric random variables.

Formally, $X$ has {\it negative hypergeometric distribution}, denoted by $X\sim \nh(N, d, r)$, if
$$\Pr{X=k} = \frac{\dbinom{k - 1}{r - 1}\dbinom{N - k}{d - r}}{\dbinom{N}{d}}\cdot$$
The expected value of $X$ is $r\frac{N + 1}{d + 1}\cdot$

We start studying the relation between geometric and negative hypergeometric random variables. 

\begin{lemma}\label{neghyper} For each $j = 1, \dots, k_1$, let $X_j$ be the waiting time until the $jth$ sucess in trials without replacement from a population of size $k_1 + k_2$ with $k_1$ possibilities of sucess (each $X_j$ has distribution $NH(k_1 + k_2, k_1, j))$. Then:
\begin{equation}\label{eq:neghyper1}
 X_1  \preceq  \geo\left( \frac{k_1}{k_1 + k_2}\right)  \preceq  \geo\left( \frac{k_1}{k_1 + k_2} - \frac{j}{k_1 + k_2}\right);\end{equation}
\begin{equation}\label{eq:neghyper2}
X_{j + 1} - X_j| X_j  \preceq  \geo\left( \frac{k_1}{k_1 + k_2} - \frac{j}{k_1 + k_2}\right).\end{equation}
\end{lemma}
\begin{proof} Since the number of failures in the population of size $k_1 +k_2$ is $k_2$, we have that
\bean \Pr{X_1 \geq m} &=& \frac{k_2}{k_1 + k_2}\,\frac{k_2 - 1}{k_1 + k_2 - 1} \cdots\frac{k_2 - (m - 2)}{k_1 + k_2 - (m - 2)}\\
 &\leq& \left( \frac{k_2}{k_1 + k_2} \right)^{m - 1} \\&=& \Pr{\geo\left(\frac{k_1}{k_1 + k_2}\right) \geq m}, \eean
where the last inequality follows observing that $\frac{k_2 - j}{k_1 + k_2 - j}\leq\frac{k_2}{k_1 + k_2}$. The second inequality in \eqnref{neghyper1} is immediate.

Now to prove \eqnref{neghyper2} observe that conditioned in $X_j = k$ the remaining population has size $k_1 + k_2 - k$ and $k_2-(k-j)$ failures. Thus,
\bean\Pr{X_{j +1} - X_j \geq m | X_j = k}
&=& \frac{k_2 - (k - j)}{k_1 + k_2 - k} \,\frac{k_2 - (k - j) - 1}{k_1 + k_2 - k -1}\cdots\frac{k_2 - (k - j) - (m - 2)}{k_1 + k_2 - k - (m - 2)}\\
&\leq& \left( \frac{k_2}{k_1 + k_2} + \frac{j}{k_1 + k_2} \right) ^{m -1}\\ &=& \Pr{\geo\left(\frac{k_1}{k_1 + k_2} - \frac{j}{k_1 + k_2}\right) \geq m},\eean
where the above inequality holds because
\bean\displaystyle \frac{k_2 - (k - j)}{k_1 + k_2 - k} &=& \frac{\frac{k_2}{k_1 + k_2}(k_1 + k_2) - k\frac{k_2}{k_1 + k_2}}{k_1 + k_2 - k} + \frac{-k\frac{k_1}{k_1 + k_2} + j}{k_1 + k_2 - k}\\ &\leq& \frac{k_2}{k_1 + k_2} + \frac{j}{k_1 + k_2}\eean and
$$ \frac{k_2 - (k - j) - l}{k_1 + k_2 - k -l} \leq \frac{k_2 - (k - j)}{k_1 + k_2 - k},$$
for $l = 1, \dots, m - 2.$\end{proof}

The next lemma enables us to dominate a negative hypergeometric random variable by a sum of geometric random variables. Its proof can be found in the appendix.

\begin{lemma}\label{cond} Let $X_1, \dots, X_j$ and $Y_1, \dots, Y_j$ be random variables taking values in $\R$ such that $X_0 \equiv 0$ and $X_{i +1} - X_i| X_i \preceq Y_{i + 1}$, for $ i = 0, \dots, j - 1$, then
$$X_j \preceq \widetilde{Y}_1 + \dots + \widetilde{Y}_j$$
where $\widetilde{Y}_i \sim Y_i$ and $\widetilde{Y}_i$ is independent of $\widetilde{Y}_k$ for $k \neq i$. \end{lemma}

We use Lemma \ref{neghyper} combined with Lemma \ref{cond} to obtain that if $X_j \sim NH(k_1 + k_2, k_1, j)$ with $j \leq k_1$, then $X_j \preceq G_1 + \dots + G_j$, where $G_1, \dots, G_j$ are independent geometric random variables with parameter $\displaystyle \frac{k_1}{k_1 + k_2} - \frac{j}{k_1 + k_2}$.

Now, consider the following standard fact about stochastic domination, which is a corollary of the Lemma \ref{cond}.

\begin{corollary}\label{ind} Let $X_1, \dots, X_r, Y_1, \dots, Y_r$ be random variables taking values in $\R$ such that $\{X_1, \dots, X_r\}$ and $\{Y_1, \dots, Y_r\}$ are independent. If $X_i \preceq Y_i$ for $i = 1, \dots, r$, then
$$X_1 + \dots + X_r \preceq Y_1 + \dots + Y_r.$$ \end{corollary}

 If $X_j \sim \nh(k_1 + k_2, k_1, j)$ and $X_l \sim \nh(k_1 + k_2, k_1, l)$ are independent random variables, we can use Lemma \ref{ind} to guarantee that $X_j + X_l \preceq G_1 + \dots + G_{j + l}$,  where $G_1, \dots, G_{j + l}$ are independent geometric random variables with parameter $ \frac{k_1}{k_1 + k_2} - \frac{\max\{j,l\}}{k_1 + k_2}$.

At this point, the problem of bounding the probability of the sum of negative hypergeometric random variables deviations above the mean is transformed into the problem of bounding the sum of geometric random variables. In the next lemma we give a bound for the sum of geometric random variables with parameter $1 - \liloh{1}$.

\begin{lemma}\label{expdec} Given any $\epsilon > 0$ and $C> 1$ there exists $\delta>0$ such that if $G_1, \dots, G_r$ are independent random variables with distribution geometric with parameter $p \geq 1 - \delta$, then
$$\Pr{G_1 + \dots + G_r > (1 + \epsilon)r} \leq \exp\left(- (C - 1) r \right).$$
\end{lemma}
\begin{proof}
Begin by taking $\epsilon > 0$ and $C > 1$. If the parameter is $1-\delta$,
\bean
\Pr{G_1 + \dots + G_r \geq (1 + \epsilon)r} &\leq& \exp\left(-\frac{C}{\epsilon}r\epsilon\right)\Ex{\exp\left(\frac{C}{\epsilon}\sum_{i = 1}^{r}(G_i - 1)\right)}\\ &=& \exp(-Cr)\left(\sum_{k = 1}^{\infty}\left(1 - \delta\right)\delta^{k - 1}\exp\left(\frac{C}{\epsilon}(k - 1)\right)\right)^r \\ &\leq& \exp(-Cr)\left(\sum_{k = 1}^{\infty}\left(\delta\exp\left(\frac{C}{\epsilon}\right)\right)^{k - 1}\right)^r.
\eean
For $\delta$ sufficiently small $\delta\exp(\frac{C}{\epsilon}) \leq \delta^{\frac{1}{2}} \leq \frac{1}{2}$ and so
\bean
\exp(-Cr)\left(\sum_{k = 1}^{\infty}\left(\delta\exp\left(\frac{C}{\epsilon}\right)\right)^{k - 1}\right)^r &=& \exp(-Cr)\left(\frac{1}{1 - \delta^{\frac{1}{2}}}\right)^r \\ &\leq& \exp(-Cr)\left(1 + 2\delta^{\frac{1}{2}}\right)^r \\
&\leq& \exp\left(-\left(C - 2\delta^{\frac{1}{2}}\right)r \right)\\
&\leq& \exp\left(-\left(C - 1\right)r \right).
\eean
This gives the result if the parameter of the geometrics is $1-\delta$. The case $p\geq 1-\delta$ follows via stochastic domination.\end{proof}

\section{Proof of \thmref{push}}

We begin by defining another process, called {\it push without replacement} and denoted by {\small PWR} or {\small PWR($G$)}: at time $t = 0$ an arbitrary vertex knows an information. In the succeeding time steps each informed vertex chooses a neighbor independently and uniformly at random from its neighbors not yet chosen and fowards the information according to this list.

Formally, let $J(t)$ be the set of vertices informed by {\small PWR} at time t. Initially $t = 0$ and $J(0) = \{v\}$, for some choice of $v \in V$. While $J(t) \neq V$, each vertex $u \in J(t)$ chooses a neighbor $v_u^t$ independently and uniformly at random in $\Gamma(u)\backslash\{v_u^s;0<s<t\}$. The new informed set is
$$J(t + 1) = J(t) \cup \{v_u^t;u \in J(t)\}.$$ The process stops when $J(t) = V$.

Define $\sJ_v (G) = \min \{t \in \mathbb{N} | J(t) = V\}$ as the first time step in which every vertex of $G$ has been informed by the {\small PWR}. The next result relates $\sT_v (G)$ and $\sJ_v(G)$. We assume that $I(0) = J(0) = \{v_n\}$ and throughout this section we denote $G = G_n, p = p_n, T = T_n, d = d_n$ and $v=v_n$, we also omit $v$ from $\sT_v (G)$ and $\sJ_v(G)$.

\begin{proposition}\label{L1} For any graph $G$ we have that $\sJ(G) \preceq \sT(G)$. Moreover, let $G_n = (V_n, E_n)$ be a sequence of $d_n$-regular graphs, with $n \rightarrow \infty$, and $T_n = \liloh{d_n}$ such that $\sJ(G_n) \leq T_n$ with high probability. Then given any $\epsilon > 0$, $\sT(G_n) \leq (1 + \epsilon) T_n$ with high probability.\end{proposition}
\begin{proof} Given any graph $G$, we start building a coupling between the push and the {\small PWR} in $G$ as follows. For each $u \in V$, let $\Theta^u = \{\Theta^u_i\}_{i =1}^{\infty}$ be a sequence of independent random variables with uniform distribution in $\Gamma(u)$. Run the push in $G$ according to the realization of $\{\Theta^u\}_{u \in V}$ in the following sense: the first neighbor to be chosen by a vertex $u$ is $\Theta^u_1$, the second is $\Theta^u_2$ and so on.

Now, for $u \in V$, define $X^u_1 = 1$ and, for each $k \in \N$,
$$X_k^u = \inf\{m \in \N| \Theta^u_m \notin \{\Theta_{X_1^u}^u, \dots, \Theta_{X^u_{k - 1}}^u\}\}.$$
Also, define $U_k^u = \Theta^u_{X^u_k}$. We will use the following fact (proof omitted):

\begin{claim}\label{claim} For an arbitrary $u \in V$, let ${\rm deg}(u) = s$. The random variable $(U^u_j)_{j = 1}^s$ has uniform distribuition in the permutations of $\Gamma(u)$. Moreover, for each $k \leq s$, $X_k^u - X_{k - 1}^u$ has geometric distribution with parameter $1 - \frac{k - 1}{s}$ and the vector $(U_j^u)_{j =1}^{s}$ and the random variables $X_1^u, X_2^u - X_1^u, \dots, X^u_s - X^u_{s - 1}$ are mutually independent.
\end{claim}
%\begin{proof}[Claim]
%Let $\{w_j|j = 1, \dots, s\} \subset \{\text{permutations of } \Gamma(u)\}$ and $\{i_k|k = 1, \dots, s\} \subset \N$, with $i_1 = 1$. Defining $X_0 = 0$ a.s., note that the joint distribution of the above random variables above is
%\bean \Pr{(U_j)_{j=1}^s = (w_j)_{j =1}^s, \bigcap_{k = 1}^{s - 1}\{X_k - X_{k -1} = i_{k +1}\}} \eean
%denoting $b_k = i_1 + \cdots + i_k$, the above probability can be written as
%\bean\Pr{\bigcap_{k =1}^s\{\Theta^u_{b_k} = w_k\}, \bigcap_{k =1}^{s -1} \bigcap_{b_k < j < b_{k + 1}}\{\Theta^u_j \in \cup_{l \leq k}\{w_{i_l}\}\}} &=& \left(\frac{1}{s}\right)^s\cdot\prod_{k -1}^{s -1}\left(\frac{k}{s}\right)^{i_{k + 1} - 1}\\ &=& \frac{1}{s!}\prod_{k = 1}^{s - 1}\left(\frac{k}{s}\right)^{i_{k + 1} - 1}\left(1 - \frac{k}{s}\right) \\&=&
%\Pr{U = (w_j)_{j =1}^s, \bigcap_{k = 1}^{s - 1}G_k = i_{k +1} },\eean
%where $U$ has uniform distribution in the permutations of $\Gamma(u)$, $G_1, \dots, G_k$ have geometric distribution with parameter $1 - \frac{k - 1}{s}$, and $U, G_1, \dots, G_k$ are independent.
%
%The proof of the Claim follows by evaluating of the marginal distributions from the joint distribution.\end{proof}

Now, run the {\small PWR} according to the realization of $\{(U_1^u, \dots, U^u_{{\rm deg}(u)})\}_{u \in V}$ in the following sense: each vertex $u$ informs its neighbors in the order given by the list $(U_1^u, \dots, U^u_{{\rm deg}(u)})$.

With the processes constructed in this way, we have $I(t) \subset J(t)$ for all $t \in \N$, which implies $\sJ(G) \preceq \sT(G)$ and we have the first part of the Lemma.

It remains to prove the second assertion. Let $v = I(0) = J(0)$ and $G$ a $d$-regular graph with $T = \liloh{d}$, then

\bean \Pr{\sT(G) > (1 + \epsilon)T} &=& \Pr{\sT(G) > (1 + \epsilon)T, \sJ(G) \leq T} \\
&& +\  \Pr{\sT(G) > (1 + \epsilon)T, \sJ(G) > T}\\
&=& \Pr{\sT(G) > (1 + \epsilon)T, \sJ(G) \leq T} + \liloh{1} \\
&=& \Pr{\exists \,w \in V; \sT_{v, w}(G) > (1 + \epsilon)T, \sJ(G) \leq T} + \liloh{1}\\
&\leq& n\cdot\Pr{ \sT_{v, w'}(G) > (1 + \epsilon)T, \sJ(G) \leq T} + \liloh{1},
\eean
where $\sT_{v, w}(G) = \min\{t \in \N| w \in I(t)\}$ is the first time step in which $w$ is informed and $w'$ is defined by the equality $\max_{w \in V}\Pr{ \sT_{v, w}(G) > (1 + \epsilon)T} = \Pr{ \sT_{v, w'}(G) > (1 + \epsilon)T}$.

Observe that
$$\mathbf{1}\{\sT_{v, w'}(G)>(1 + \epsilon)T, \sJ(G)\leq T\} = f(\{X^u_k - X^u_{k - 1}\}_{k = 1, u \in V}^{k = d}, \{(U_{X_k^u})_{k = 1}^d\}_{u \in V})$$
where $f$ is a measurable function. Recall that $\{X^u_k - X^u_{k - 1}| 1 \leq k \leq d, u \in V\}$ and $\{(U_{X_k^u})_{k = 1}^d| u \in V\}$ are mutually independent, then by the substitution principle, we have that
\bean \Pr{ \sT_{v, w'}(G) > (1 + \epsilon)T, \sJ(G) \leq T} &=& \Ex{\Ex{f(\{X^u_k - X^u_{k - 1}\}_{k, u}, \{(U_{X_k^u})_k^d\}_{u})|\{(U_{X_k^u})_{k}^d\}_{u}}}\\ &=&\Ex{\phi(U_{X_k^u})_k\}_u)},\eean
%&=& \Ex{\Ex{f(\{X^u_k - X^u_{k - 1}\}_{k, u}, \{(x^u_k)_k\}_u)}\bigg|_{\{(x^u_k)_k\}_u=\{(U_{X_k^u})_k\}_u}}.\\
%\eean
where $\phi(\{(x^u_k)_k\}_u = \Ex{f(\{X^u_k - X^u_{k - 1}\}_{k, u}, \{(x^u_k)_k\}_u)}$.

Fixed the realization $\{(U_{X_k^u})_k\}_u = \{(x^u_k)_k\}_u$, the event $\{\sJ(G)\leq T\}$ implies the existence of a path $\gamma: v = v_0\sim\dots\sim v_l=w'$ such that $J_{v_0, v_1} + \dots + J_{v_{l - 1}, v_l} \leq T$, where $J_{v_j, v_{j +1}}$ is the first time step in which $v_j$ choose $v_{j +1}$ to transmit the information in the {\small PWR}. As $\{\sT_{v, w'}(G) > (1 + \epsilon)T\}$, defining $T_{v_j, v_{j + 1}}$ as the first time step in which $v_j$ choose $v_{j +1}$ to transmit the information in the push in $G$, we have that $T_{v_0, v_1} + \dots + T_{v_{l - 1}, v_l} \geq (1 + \epsilon)T$. So,
$$\Ex{f(\{X^u_k - X^u_{k - 1}\}_{k, u}, \{(x^u_k)_k\}_u)} \leq \Pr{T_{v_0, v_1} + \dots + T_{v_{l - 1}, v_l} \geq (1 + \epsilon)T}.$$

To bound the probability above we find the distribution of $T_{v_j, v_{j + 1}}$,
\bean
T_{v_j, v_{j + 1}} &=& X^{v_j}_{J_{v_j, v_{j + 1}}} \\
&=& X^{v_j}_1 + (X^{v_j}_2 - X^{v_j}_1) + \dots + (X^{v_j}_{t_{v_j, v_{j + 1}}} - X^{v_j}_{t_{v_j, v_{j + 1}} - 1})\\
&\sim& G^j_1 + \dots + G^j_{J_{v_j, v_{j + 1}}}
\eean
where $G^j_1, \dots , G^j_{J_{v_j, v_{j + 1}}}$, by the Claim \ref{claim}, are independent random variables such that $G^j_k$ has geometric distribution with parameter $1 - \frac{k - 1}{d}$, for each $k = 1, \dots, J_{v_j, v_{j + 1}}$.

As $\geo(1 - \frac{k - 1}{d}) \preceq \geo(1 - \frac{T}{d})$ and $J_{v_0, v_1} + \dots + J_{v_{l - 1}, v_l} \leq T$, using Lemma \ref{ind} we have
$$T_{v_0, v_1} + \dots + T_{v_{l - 1}, v_l} \preceq G_1' + \dots + G_T',$$
where $G_1', \dots,G_T'$ are independent geometric random variables with parameter $1 - \frac{T}{d}$. Then,
\begin{equation}\label{geo}
\Pr{T_{v_0, v_1} + \dots + T_{v_{l - 1}, v_l} \geq (1 + \epsilon)T} \leq \Pr{G_1' + \dots + G_T' \geq (1 + \epsilon)T}.
\end{equation}

Now, using Lemma \ref{expdec} with $C> 1 + \ln 2$ and $n$ sufficiently large for $\frac{T}{d} < \delta$,
\bean
\Pr{G_1' + \dots + G_T' \geq (1 + \epsilon)T} \leq \exp\left(-(C - 1)T\right),
\eean
as $T > \log_2n$ we have that
\bean
\Pr{G_1' + \dots + G_T' \geq (1 + \epsilon)T} &\leq& \exp\left(-(C - 1)\frac{\ln n}{\ln 2}\right)\\
&=& \liloh{n^{-1}}.
\eean

Thus, $n\cdot\Pr{ \sT_{v, w'}(G) > (1 + \epsilon)T, \sJ(G) \leq T} = \liloh{1}$ and the proof is finished. \end{proof}

The next lemma enables us to build a coupling between the {\small PWR} in $G$ and the {\small PWR} in $G_p$.

\begin{lemma}\label{coupling} Let $G_n$ be a sequence of $d_n$-regular graphs, $T_n$ and $p_n$ as in \thmref{push}. Take $C=\left(\frac{pd}{T}\right)^{\frac{1}{2}}$, then, for each $u \in V$, there exist random variables $N_u^p$ with distribution $Bin(d, \frac{CT}{d})$, $N_u$ with distribution $Bin(d, \frac{CT}{d}(1-\frac{CT}{pd}))$, and random elements $\sN_u$ in $\Gamma(u)$, $\sN_u^p$ in $\Gamma_p(u)$ such that:
\bea\label{PI}\Pr{\sN_u = S|N_u = k} = \frac{1}{\dbinom{d}{k}}, \text{ for each } S \subset \Gamma(u) \text{ with } |S| = k,\eea
and
\bea\label{PII}\Pr{\sN_u^p = S'|N_u^p = j} = \frac{1}{\dbinom{deg_p(u)}{j}}, \text{ for each } S' \subset \Gamma_p(u) \text{ with } |S'| = j.\eea
Moreover, $\sN_u \subset \sN_u^p$ and there exists $\delta =\delta_n\rightarrow 0$ as $n\rightarrow \infty$ such that
\bea\label{PIII}\Pr{|\sN_u^p\backslash \sN_u| \leq \delta\cdot|\sN_u|,\forall u\in V} = 1 - \liloh{n^{-1}}.\eea\end{lemma}
\begin{proof}
Begin by taking independent random variables $\{I'_{u\rightarrow v}|(u,v)\in V^2,u\sim v\}$ with distribution $Be(\frac{CT}{pd})$, also take, for each $e\in e(G)$, independent random variables $A_e$ with distribution $Be(p)$ (independent of the $I'$s).

Now defining, for each $(u, v)\in V^2$ with $u\sim v$, $I^p_{u\rightarrow v}=A_eI'_{u\rightarrow v}$ and, for each $u\in V$, $\sN_u^p=\{v\sim u|I^p_{u\rightarrow v}=1\}$ and $N^p_u=|\sN_u^p|$, we have that $I^p_{u\rightarrow v}$ has distribution $Be(\frac{CT}{d})$ and therefore $N_u^p$ has distribution $Bin(d, \frac{CT}{d})$ and $\sN_u^p$ satisfies (\ref{PII}).

We will use the following general fact to build $N_u$ and $\sN_u$:\\

\noindent{\bf Strassen's Lemma} {\it Let $\mu$ and $\nu$ be distributions on $\R^2$ such that $\mu\preceq\nu$. Then, there exists a coupling $(X, Y)$ of the $(\mu, \nu)$ such that $X\leq Y$ in the coordinatewise partial order.}\\

Let $\mu$ be the law of $(I^p_{u\rightarrow v}, I^p_{v\rightarrow u})$ and $\nu$ the distribution of $(B_1, B_2)$ where $B_1, B_2$ are independent random variables $Be(q)$ such that $q= \frac{CT}{d}(1-\frac{CT}{2pd})$. To show that $\nu\preceq\mu$, by the symmetry of the distributions, it suffices to show that $S_2\preceq S_1$, where $S_1=I^p_{u\rightarrow v}+I^p_{v\rightarrow u}$ and $S_2=B_1+B_2$.

Our choice of $q$ implies, by simple calculations, that $\Pr{S_2\geq2}\leq\Pr{S_1\geq2}$ and $\Pr{S_2\geq1}\leq\Pr{S_1\geq1}$ so, $S_2\preceq S_1$. Now, using Strassen's Lemma, we obtain that there exists $(I_{u\rightarrow v}, I_{v\rightarrow u})$ with distribution $Be(q)\times Be(q)$ and such that $I_{u\rightarrow v}\leq I^p_{u\rightarrow v}$ and $I_{v\rightarrow u}\leq I^p_{v\rightarrow u}$. Moreover, as
$$\Pr{I^p_{u\rightarrow v}-I_{u\rightarrow v}=1}=\Ex{I^p_{u\rightarrow v}-I_{u\rightarrow v}}=\frac{1}{2p}\left(\frac{CT}{d}\right)^2$$
we have that $(I^p_{u\rightarrow v} - I_{u\rightarrow v})$ has distribution $Be\left(\frac{1}{2p}\left(\frac{CT}{d}\right)^2\right)$.

Defining, for each $u\in V$, $\sN_u=\{v\sim u|I_{u\rightarrow v}=1\}$ and $N_u=|\sN_u|$. It immediately follows that $N_u$ has distribution $Bin\left(d, \frac{CT}{d}\left(1-\frac{CT}{2pd}\right)\right)$, $\sN_u^p$ satisfies (\ref{PI}) and $\sN_u \subset \sN_u^p$. It remains to prove (\ref{PIII}).

Let us prove (\ref{PIII}). Start by choosing $\delta=\max\left\{C^{-\frac{1}{2}},\left(\frac{CT}{pd}\right)^{\frac{1}{2}}\right\}$ and defining $N_u^*=|\sN_u^p\backslash \sN_u|$ that has distribution $Bin\left(d, \frac{1}{2p}\left(\frac{CT}{d}\right)^2\right)$. Let $A$ be the event $\{|N_u-\Ex{N_u}|\leq\frac{1}{2}\Ex{N_u}\}$. As $T\geq\log_2n$ and $C\rightarrow\infty$ we have that $\Ex{N_u}\gg\ln n$. So, we can use Chernoff bounds for the binomial distribution to obtain $\Pr{A^c}=\liloh{n^{-2}}$. Thus,
\bean\Pr{N_u^*>\delta N_u}&=&\Pr{N_u^*>\delta N_u,A}+\Pr{A^c}\\
&\leq&\Pr{N_u^*>\delta \frac{3\Ex{N_u}}{2}} + \liloh{n^{-2}}\\
&=&\liloh{n^{-2}}
\eean
where in the last inequality we use Chernooff bounds again and our choice of $\delta$ which ensures $\delta\cdot\Ex{N_u}\gg\max\{\ln n,\Ex{N_u^*}\}$. This finishes the proof of Lemma.
\end{proof}

The next result studies the relation between $\sJ(G)$ and $\sJ(G_p)$.

\begin{proposition}\label{L2} Let $G_n = (V_n, E_n)$ be a sequence of $d_n$-regular graphs, with $n \rightarrow \infty$, that satisfies $\sJ(G_n) \leq T_n$ with high probability. Then, given $\epsilon > 0$ and choosing $p_n = \liloh{1}$ such that $T_n = \liloh{p_n\cdot d_n}$, we have that $\sJ(G_{n,{p_n}}) \leq (1 + \epsilon)T_n$ with high probability. \end{proposition}

\begin{proof} We begin building a coupling between the {\small PWR} in the graph $G$ and the {\small PWR} in the graph $G_p$ using the previous lemma as follows. Take random variables $N_u$, $N_u^p$ and random elements $\sN_u$, $\sN_u^p$ as in Lemma \ref{coupling} and fix an arbitrary order $\{1, \dots, n\}$ of the vertices of $G$. For each $u \in V$, order $\Gamma(u)$ in the following way:
\begin{itemize}
    \item order $\sN_u$ uniformly (from 1 to $N_u$);
    \item order $\Gamma(u)\backslash\sN_u$ uniformly (from $N_u + 1$ to $deg(u)$).
\end{itemize}

Equation (\ref{PI}) ensures that this is a uniform ordering of the neighbors of $u$. Denote by $Ord_u = Ord_u(w^u_1, \dots,w^u_d)$, where $w_1^u<\dots<w_d^u$ are the neighbors of $u$, the random vector built by the manner above. Run the {\small PWR$(G)$} as follows. If $t$ is the first time that the vertex $u$ receives information then in step $t+1$ the vertex $u$ informs $w_{i_1}^u$ if the $i_1-$th coordinate of $Ord_u(w^u_1, \dots,w^u_d)$ is equal to $1$, in step $t+2$ the vertex $u$ informs $w_{i_2}^u$ if the $i_2-$th coordinate of $Ord_u(w^u_1, \dots,w^u_d)$ is equal to $2$ and so on.

Let us run the {\small PWR$(G_p)$} similarly. For each $u \in V$, order $\Gamma_p(u)$ in the following way:
\begin{itemize}
    \item order $\sN^p_u$ uniformly, but conditioned to coincide with the order of $\sN_u$ (from 1 to $N^p_u$);
    \item order $\Gamma(u)\backslash\sN^p_u$ uniformly (from $N^p_u + 1$ to $deg_p(u)$).
\end{itemize}

Equation (\ref{PII}) ensures that this is a uniform ordering of the neighbors of $u$. Denote by $Ord_u^p = Ord_u^p(\tilde{w}^u_1, \dots,\tilde{w}^u_{deg_p(u)})$, where $\tilde{w}_1^u<\dots<\tilde{w}_{deg_p(u)}^u$ are the neighbors of $u$ in $\Gamma_p(u)$, the random vector built by the manner above. Run the {\small PWR$(G_p)$} using $\{Ord_u^p\}_{u \in V}$.

Now we will prove that if $\sJ(G)\leq T$ with high probability then $\sJ(G_p)\leq (1+\epsilon)T$ with high probability. Take $\epsilon > 0$, we have that
\bea\label{eqb}\Pr{\sJ(G_p) > (1+ \epsilon)T}  =
 \Pr{\sJ(G_p) > (1+ \epsilon)T, A} + \Pr{\sJ(G_p) > (1+ \epsilon)T, A^c},\eea
where $A$ is the event $\{\sJ(G) \leq T\}\cap\{N_u > \frac{CT}{2}, |\sN_u^p\backslash \sN_u| \leq \delta\cdot|\sN_u|, N_u^p>\frac{CT}{2}, \,\forall u\in V\}$ with $C = \left(\frac{pd}{T}\right)^{\frac{1}{2}}$ and $\delta=\delta_n$ as in Lemma \ref{coupling}. First, we analyze the term $\Pr{\sJ(G_p) > (1 + \epsilon)T, A^c}$:
\bean \Pr{\sJ(G_p) > (1 + \epsilon)T, A^c} &\leq& \Pr{\sJ(G) \geq T}\\ && +\, \Pr{\exists u ;N_u \leq \frac{CT}{2}}\\&& +\, \Pr{\exists u; N_u^p\leq\frac{CT}{2}}\\&& +\, \Pr{\exists u;|\sN_u^p\backslash \sN_u| > \delta\cdot|\sN_u|}.
\eean
By hypothesis $\Pr{\sJ(G) \geq T}=\liloh{1}$. Using Chernoff bounds for the binomial distribution, $T\geq \log_2n$ and $C \rightarrow\infty $ we have that $\Pr{N_u \leq \frac{CT}{2}} = \liloh{n^{-1}}$ and that
$\Pr{N^p_u \leq \frac{CT}{2}} = \liloh{n^{-1}}.$
So,
\bean \Pr{\sJ(G_p) > (1 + \epsilon)T, A^c} = \liloh{1}.\eean

Next, let us estimate the other term in equation (\ref{eqb}). For this, let $\sJ_{v, u}(G) = \min\{t \in \N| u \in J(t)\}$ be the first time step in which $u$ is informed by {\small PWR$(G)$} and let $w'$ be the vertex such that $$\Pr{\sJ_{v,w'}(G_p) > (1+ \epsilon)T,  A} = \max_{u \in V} \Pr{\sJ_{v,u}(G_p) > (1+ \epsilon)T,  A}.$$
Then,
\bean\Pr{\sJ(G_p) > (1+ \epsilon)T, A}
&\leq& \Pr{\exists  u \in V ;  \sJ_{v,u}(G_p) > (1+ \epsilon)T,  A}\\
&\leq&  n\cdot\Pr{\sJ_{v,w'}(G_p) > (1+ \epsilon)T,  A}, \eean
and the last expression is equal to
\bea\label{expcond}n\cdot\Ex{\mathbf{1}_A\mathbb{P}( \sJ_{v,w'}(G_p) > (1 + \epsilon)T|\{Ord_u\}_{u \in V}, \{N_u\}_{u \in V}, \{N^p_u\}_{u \in V})},\eea
because $A$ is $\sigma(\{Ord_u\}_{u \in V}, \{N_u\}_{u \in V},\{N^p_u\}_{u \in V})$-measurable.

Defining
$$ \mathbb{Q}(\cdot) = \mathbb{P}(\cdot|  \{Ord_u\}_{u \in V} = \{(m_1^u, \dots, m_d^u)\}_{u \in V}, \{N_u\}_{u \in V} = \{n_u\}_{u \in V},\{N^p_u\}_{u \in V} = \{n^p_u\}_{u \in V})$$
and $\mu$ as the law of $(\{Ord_u\}_{u \in V}, \{N_u\}_{u \in V},\{N^p_u\}_{u \in V})$, we obtain that expression (\ref{expcond}) is equal to
\bea\label{int}n\cdot \int \mathbf{1}_{\widetilde{A}}(\{\widetilde{m}_u\}_u,\{n_u\}_u,\{n^p_u\}_u)\cdot\mathbb{Q}(\sJ_{v,w'}(G_p) > (1+ \epsilon)T)\, d\mu(\{\widetilde{m}_u\}_u,\{n_u\}_u,\{n^p_u\}_u).\eea
where $\widetilde{m}_u = (m_1^u, \dots, m_d^u)$ and $\widetilde{A}$ is a Borel measurable set such that $\mathbf{1}_A$ is equal to $\mathbf{1}_{\widetilde{A}}(\{Ord_u\}_u,\{N_u\}_u,\{N^p_u\}_u)$.

For $u \sim w$, let $J_{u, w}(G)$ be the time until vertex $u$ chooses $w$ to transmit the information according to the {\small PWR$(G)$} and $J_{u, w}(G_p)$ the analogue for {\small PWR$(G_p)$}. Fixed $\{Ord_u=\widetilde{m}_u\}_{u \in V}$, the event $\{\sJ(G) \leq  T\}$ implies the existence of a path $\gamma : v = v_0 \sim \dots \sim v_l = w'$ such that $J_{v_0, v_1}(G) + \dots + J_{v_{l - 1},v_l}(G) \leq T$. Also note that $N_{v_j}>T$ and $J_{v_j, v_{j+1}}\leq T$ implies $(v_j, v_{j+1})\in G_p$ for each $j=0,\dots, l-1$. So, expression (\ref{int}) is less than or equal to
\bea\label{int}n\cdot \int \mathbf{1}_{\widetilde{A}}(\{\widetilde{m}_u\}_u,\{n_u\}_u,\{n^p_u\}_u)\cdot\mathbb{Q}(J_{v_0, v_1}(G_p) + \dots + J_{v_{l - 1},v_l}(G_p) > (1+ \epsilon)T)\, d\mu.\eea

Now, we will find the conditional distribution of $J_{v_j,v_{j+1}}(G_p)$. Take $0\leq r\leq m\leq l$ and $0\leq k\leq l$, we want to calculate
\bean \Pr{J_{v_j,v_{j+1}}(G_p)=k|J_{v_j,v_{j+1}}(G)=r, N_{v_j}=m, N^p_{v_j}=l}.\eean
As $r\leq m$, we have that $v_{j+1}\in \sN_{v_j}\subset\sN^p_{v_{j}}$. Moreover, $J_{v_j,v_{j+1}}(G)=r$ implies that $v_{j+1}$ is the $(j+1)$-th vertex chosen in uniform ordering of $\sN_{v_j}$. As the order of $\sN^p_{v_j}$ preserves the order of $\sN_{v_j}$, we have that $J_{v_j,v_{j+1}}(G_p)$ is the time until the $(j+1)$-th sucess is obtained in repetead random sampling without replacement from a dichotomous population of size $N^p_{v_j}$ with $N_{v_j}$ successes. Then, in $\{J_{v_j,v_{j+1}}(G) \leq N_{v_j}\}$,
\bean \Pr{J_{v_j,v_{j+1}}(G_p)=k|J_{v_j,v_{j+1}}(G), N_{v_j}, N^p_{v_j}}= \Pr{NH(N_{v_j}^p,N_{v_j},J_{v_j,v_{j+1}}(G))=k}.\eean
So, the expression (\ref{int}) can be written as
\bea\label{int2}n\cdot \int \mathbf{1}_{\widetilde{A}}(\{\widetilde{m}_u\}_u,\{n_u\}_u, \{n^p_u\}_u)\cdot\Pr{Y_0 + \dots + Y_l > (1+ \epsilon)T}\, d\mu\eea
where, for each $j=0,\dots,l-1$, $Y_j$ has distribution $NH(N_{v_j}^p,N_{v_j},J_{v_j,v_{j+1}}(G))$ and $J_{v_0,v_1} + \dots + J_{v_{l-1},v_l} \leq T$ in $\widetilde{A}$.

Inside $\{\forall u\in V;|\sN_u^p\backslash \sN_u| \leq \delta\cdot|\sN_u|, N_u^p>\frac{CT}{2}\}$ we have that
\bean \frac{N_u}{N_u^p}-\frac{J_{v_j,v_{j+1}}(G)}{N_u^p}&\geq&\frac{N_u}{N_u+|\sN_u^p\backslash\sN_u|}-\frac{T}{N_u^p}\\
&\geq&\frac{1}{1+\delta}-\frac{2T}{CT}\\&=&1-\liloh{1}.
\eean
Then, using previous section
$$\Pr{Y_0 + \dots + Y_l > (1+ \epsilon)T}\leq\Pr{G_0 + \dots + G_T > (1+ \epsilon)T}$$
where $G_0, \dots, G_T$ are independent geometric random variables with parameter $1-\liloh{1}$. Thus, for $n$ sufficiently large, we can use Lemma \ref{expdec} and obtain that
$$\Pr{G_0 + \dots + G_T > (1+ \epsilon)T} = \liloh{n^{-1}},$$
because $T\geq \log_2n$. So, expression (\ref{int2}) is bounded by $\liloh{1}$ and this finishes the proof of the proposition.
\end{proof}

\begin{proof}[\thmref{push}] By hypothesis, $\sT(G_n) \leq T_n$ with high probability. Then, using Proposition \ref{L1}, we have that $\sJ(G_n) \leq T_n$ w.h.p.. As $T_n = \liloh{p_n.d_n}$ and $p_n=\liloh{1}$ we can use Proposition \ref{L2} and obtain that $\sJ(G_{n, p_n}) \leq (1 + \epsilon')T_n$ w.h.p.. Finally, we use the second part of Proposition \ref{L1} to conclude that $\sT(G_{n, p_n}) \leq (1+\epsilon')^2T_n$ w.h.p.. Choosing $\epsilon'$ appropriately we have the result.\end{proof}

\section{Appendix}

In this appendix we prove Lemma \ref{cond}. First, consider the following theorem (see \cite{Kll}):

\begin{theorem}[Disintegration]\label{disintegration} Let $\xi$ and $\eta$ be two random variables in a probability space $(\Omega, \sF, \Prwo)$. Consider a measurable function $f$ on $\R\times\R$ with $\Ex{|f(\xi, \eta)|} < \infty$. Then
$$\Ex{f(\xi, \eta)|\eta} = \int f(s, \eta) \,\mu(\eta, ds),\, \Prwo-a.s.,$$
where $\mu(\eta, \cdot)$ is the regular conditional probability of $\xi$ given $\eta$. \end{theorem}

\begin{proof}[Lemma \ref{cond}]
We proceed by induction in $r$.

For $r = 1$ the result follows because $X_1 - X_0|X_0 \sim X_1$ and $X_1 \preceq Y_1$.

Now, define $f(\xi, \eta) = 1\{\xi + \eta \geq t\}$. Hence
\bean\Pr{X_r \geq t} &=& \Pr{X_{r -1} + X_r - X_{r -1}\geq t}\\ &=& \Ex{f(X_{r - 1}, X_r - X_{r -1})}\\ &=& \Ex{\Ex{f(X_{r - 1}, X_r - X_{r -1})| X_{r - 1}}}\eean

Let $\mu(X_{r - 1}, \cdot)$ be the regular conditional probability of $ X_r - X_{r -1}$ given $X_{r - 1}$ and $\nu$ the distribution of $Y_r$. By Theorem \ref{disintegration}, the expression above is equal to
$$\Ex{\int f(X_{r - 1}, s)\, \mu (X_{r -1}, ds)}.$$

By assumption, we have $X_r - X_{r -1}| X_{r - 1} \preceq Y_r$ and this means that $\mu (X_{r - 1} , \cdot) \preceq \nu, \ \Prwo$-almost surely. As $f(\xi, \cdot)$ is increasing upper semicontinuos we have that
$$\Ex{\int f(X_{r - 1}, s)\, \mu (X_{r -1}, ds)} \leq \Ex{\int f(X_{r -1}, s)\, \nu (ds)}.$$

Now, take $\widetilde{Y}_r \sim Y_r$ independent of $X_{r - 1}$ and let $\vartheta$ be the distribution of $\widetilde{Y}_r$. The last expression is equal to
$$ \Ex{\int f(X_{r -1}, s) \,\vartheta (ds)} = \Pr{X_{r - 1} + \widetilde{Y}_r \geq t}.$$
Since $\widetilde{Y}_r$ is independent of $X_{r - 1}$, we can use the substitution principle to obtain
\bean \Pr{X_{r - 1} + \widetilde{Y}_r \geq t} &=& \Ex{\Ex{f(X_{r - 1}, \widetilde{Y}_r)| \widetilde{Y}_r}}\\
&=& \Ex{\Ex{f(X_{r - 1}, \tilde{y}_r)}(\tilde{y}_r = \widetilde{Y}_r)}\\ &=& \Ex{\Pr{X_{r - 1} \geq t - \tilde{y}_r}(\tilde{y}_r = \widetilde{Y}_r)}.\eean

By the induction hypothesis, there exist $\widetilde{Y}_i \sim Y_i$, for $i = 1, \dots, r-1$, independent random variables such that $X_{r - 1} \preceq \widetilde{Y}_1 + \dots + \widetilde{Y}_{r - 1}$ (we can take $\widetilde{Y}_1, \dots, \widetilde{Y}_{r - 1}$ independent of $\widetilde{Y}_r$). Therefore,
\bean\Pr{X_r \geq t} &\leq& \Pr{X_{r - 1} + \widetilde{Y}_r \geq t}\\
&\leq& \Pr{\widetilde{Y}_1 + \dots + \widetilde{Y}_{r - 1} + \widetilde{Y}_r \geq t}\eean
and the result follows.
\end{proof}

\textbf{Acknowledgments} Alan Prata thanks Bruno Gois, Marcelo Hil\'ario and Yuri Lima for useful suggestions.

\noindent (Roberto Imbuzeiro Oliveira) IMPA, Rio de Janeiro\\
{\it E-mail address:} rimfo@impa.br\\ \\
(Alan Prata) IMPA, Rio de Janeiro\\
{\it E-mail address:} alan@impa.br

\end{document}